\documentclass[a4paper]{amsart}
\usepackage{amssymb}
\usepackage{stmaryrd}
\usepackage[all]{xy}
\usepackage{epsf}
\usepackage{enumerate}
\newcommand{\printname}[1] {}

\newtheorem{theorem}[]{Theorem}

\newtheorem{lemma}[]{Lemma}

\newtheorem{definition}[]{Definition}
\newtheorem{example}[theorem]{Example}
\theoremstyle{remark}

\newcommand{\rmap}{\longrightarrow}

\usepackage{amsmath,amscd}
\usepackage{wrapfig}

\begin{document}
\title{On the existence of symplectic realizations}
\author{Marius Crainic}
\address{Depart. of Math., Utrecht University, 3508 TA Utrecht,
The Netherlands}
\email{crainic@math.uu.nl}
\author{Ioan M\v arcu\c t}
\address{Depart. of Math., Utrecht University, 3508 TA Utrecht,
The Netherlands}
\email{I.T.Marcut@uu.nl}
\begin{abstract}
We give a direct global proof for the existence of symplectic realizations of arbitrary Poisson manifolds.
\end{abstract}
\maketitle

\setcounter{tocdepth}{1}

\section*{Introduction}

Let $(M, \pi)$ be a Poisson manifold. A \textbf{symplectic
realization} of $(M, \pi)$ is a symplectic manifold $(S, \omega)$
together with a Poisson submersion
\[ \mu: (S, \omega) \rmap (M, \pi).\]
Although the existence of symplectic realizations if a fundamental
result in Poisson geometry, the known proofs are rather involved.
Originally, the local result was proven in \cite{Wein} and a gluing
argument was provided in \cite{DWC}; the same procedure appears in
\cite{Kar}. The path approach to symplectic groupoids \cite{CaFe,
CrFe2} gives a different proof. Here we present a direct, global,
finite dimensional proof, based on the philosophy of
\textbf{contravariant geometry}: in Poisson geometry, the relevant
tangent directions come from the cotangent bundle $T^*M$ via the
bundle map
\[ \pi^{\sharp}: T^*M \rmap TM\]
which is just $\pi$ converted into a linear map ($\beta(\pi^{\sharp}(\alpha))= \pi(\alpha, \beta)$). More on contravariant geometry can be found in the next section.
We will use a contravariant version of the notion of spray. In the following definition, for $t> 0$,
we denote by $m_t: T^*M\rmap T^*M$ the fiberwise multiplication by $t$.

\begin{definition} A \textbf{Poisson spray} on the Poisson manifold $(M,\pi)$ is a vector field $\mathcal{V}_{\pi}$ on $T^*M$ satisfying the following two properties:
\begin{enumerate}
\item[(1)] $(dp)_{\xi}(\mathcal{V}_{\pi, \xi})= \pi^{\sharp}(\xi)$ for all $\xi\in T^*M$.
\item[(2)] $m_{t}^{*}(\mathcal{V}_{\pi}) = t\mathcal{V}_{\pi}$ for all $t> 0$.
\end{enumerate}
We denote by $\varphi_t$  the flow of $\mathcal{V}_{\pi}$.
\end{definition}

A short discussion on Poisson sprays-completely analogous the classical sprays \cite{Lang}- is given in the next section.
Condition (1) means that the integral curves of $\mathcal{V}_{\pi}$ are cotangent curves (see the next section);
it also appears in \cite{Wein2} under the name ``second order differential equation''.
Our main result is the following.

\begin{theorem} \label{main-theorem} Given the Poisson manifold $(M, \pi)$ and a contravariant spray $\mathcal{V}_{\pi}$,
there exists an open neighborhood $\mathcal{U}\subset T^*M$ of the
zero-section so that
\[ \omega:= \int_{0}^{1} (\varphi_t)^*\omega_{\textrm{can}} dt\]
is a symplectic structure on $\mathcal{U}$ and the canonical projection
$p: (\mathcal{U}, \omega)\rmap (M, \pi)$
is a symplectic realization.
\end{theorem}

\begin{example} \rm \ When $M= U\subset \mathbb{R}^n$ open, denoting by $\pi_{i, j}$ the components of $\pi$, the simplest contravariant spray is
$\mathcal{V}_{\pi}(x, y)= \sum_{p, q} \pi_{p, q}(x) y_p\frac{\partial}{\partial x_q}$, where $x$ are the coordinates of $U$ and $(x, y)$ the induced coordinates
on $T^*U$. It is not difficult to see that the resulting $\omega$ coincides with the one constructed by A.Weinstein \cite{Wein}.
\end{example}

One may expect that the proof is ``just a computations''. Although
that is true in principle, the computation is more subtle then one
may believe. In particular, we will make use of the principle of
``contravariant geometry'' which is intrinsic to Poisson geometry.
The fact that the proof cannot be so trivial and hides some
interesting geometry behind was already observed in the local case
by A. Weinstein in \cite{Wein}: the notion of contravariant spray,
its existence, the formula for $\omega$ (giving a symplectic form on
an small enough $\mathcal{U}$)- they all make sense for any bivector
$\pi$, Poisson or not. But the fact that the push-down of (the
inverse of) $\omega$ is $\pi$ can only hold for Poisson bivectors.
Nowadays, with all the insight into symplectic groupoids, we can say
that we have the full geometric understanding of this theorem; in
particular, it can be derived from the path-approach to symplectic
groupoids of \cite{CaFe} and the resulting construction of local
symplectic groupoids \cite{CrFe2}. However, it is clearly worth
giving a more direct, global argument.

Let us start already with the first steps of the proof. Let's first
look at $\omega$ on vectors tangent to $T^*M$ at zero's $0_x\in
T^{*}_{x}M$ ($x\in M$). At such points one has a canonical
isomorphism $T_{0_x}(T^*M)\cong T_xM\oplus T_{x}^{*}M$ denoted
$v\mapsto (\overline{v}, \theta_v)$, and the canonical symplectic
form is
\begin{equation}\label{omega-can-0}
\omega_{\textrm{can}, 0_x}(v, w)= \langle \theta_{w},
\overline{v}\rangle -\langle \theta_{v}, \overline{w}\rangle.
\end{equation}
From the properties of $\mathcal{V}_{\pi}$ it follows that $\varphi_t(0_x)= 0_x$ for all $t$ and all $x$, hence $\varphi_{t}$ is well defined on a neighborhood of the zero-section, for all $t\in [0, 1]$.
From the same properties it also follows that
\[ (d\varphi_t)_{0_x} : T_{0_x}(T^*M)\rmap T_{0_x}(T^*M) \]
is, in components,
\[ (\overline{v}, \theta_v)\mapsto (\overline{v}+ t\pi^{\sharp}\theta_{v}, \theta_{v}) .\]
From the definition of $\omega$ and the previous formula for $\omega_{\textrm{can}}$ we deduce that
\begin{equation}\label{omega-0}
\omega_{0_x}(v, w)= \langle \theta_{w}, \overline{v}\rangle- \langle
\theta_{v}, \overline{w}\rangle+ \pi(\theta_{v}, \theta_{w})
\end{equation}
for all $v, w\in T_{0_x}(T^*M)$. This implies that $\omega$ is nondegenerate at all zero's $0_x\in T^*M$. Hence we can find a neighborhood
$\mathcal{U}$ of the zero-section in $T^*M$ such that $\varphi_t$ is defined on $\mathcal{U}$ for all $t\in [0, 1]$ and $\omega|_{\mathcal{U}}$ is nondegenerate
(hence symplectic). Fixing such an $\mathcal{U}$, we still have to show that the map
\[ (dp)_{\xi}: T_{\xi}(\mathcal{U})\rmap T_{p(\xi)}M \]
sends the bivector associated to $\omega$ to $\pi$, for all $\xi\in
\mathcal{U}$. The fact that this holds at all $\xi= 0_x$ follows
immediately from the previous expression for $\omega_{0_x}$. Our job
is to show that it holds at all $\xi$'s. Although it will not be
used in this paper, it is worth mentioning here Libermann's result
on symplectically complete foliation, concerning the following
question: given a symplectic manifold $(\mathcal{U}, \omega)$ and a
submersion $p: \mathcal{U}\rmap M$, when can one push-down the
bivector associated to $\omega$ to a bivector on $M$? Libermann's
theorem (see e.g. Theorem 1.9.7 in \cite{DZ}) gives us the following
characterization: considering the involutive distribution
$\mathcal{F}(p)\subset T\mathcal{U}$ tangent to the fibers of $p$,
its symplectic orthogonal with respect to $\omega$,
$\mathcal{F}(p)^{\perp}\subset T\mathcal{U}$, must be involutive.
What happens in our case is the following:
\begin{equation}\label{main-lemma}
\mathcal{F}(p)^{\perp}= \mathcal{F}(p_1),
\end{equation}
where $\mathcal{F}(p_1)$ is the (involutive!) distribution tangent to the fibers of $p_1:= p\circ \varphi_1: \mathcal{U}\rmap M$.
This will be proven in the last section. However, it turns out that the ingredients needed to prove this equality can be used
to show directly that $p$ is a Poisson map, without having
to appeal to Libermann's result.

\section{Contravariant geometry}

As we have already mentioned, the basic idea of contravariant geometry in Poisson geometry is that of replacing the tangent bundle $TM$ by $T^*M$. The two are related
by the bundle map $\pi^{\sharp}$.  But the main structure that makes everything work is the presence of a Lie bracket $[-, -]_{\pi}$ on $\Gamma(T^*M)$, which is the contravariant analogue of the Lie bracket on vector fields
(the two brackets being related via $\pi^{\sharp}$). It is uniquely determined by the condition
\[ [df, dg]_{\pi}=d\{f, g\}\]
and the Leibniz identity
\[ [\alpha, f\beta]_{\pi}= f[\alpha, \beta]_{\pi}+ L_{\alpha}(f) \beta \]
for all $\alpha, \beta\in \Omega^1(M)$, where $L_{\alpha}:= L_{\pi^{\sharp}\alpha}$ is the Lie derivative along the ordinary vector field $\pi^{\sharp}(\alpha)$
associated to $\alpha$. In other words, contravariant geometry is the geometry associated to the Lie algebroid $(T^*M, \pi^{\sharp}, [-, -]_{\pi})$.
Here are some examples of notions that are contravariant to the usual ones (see e.g. \cite{CrFe2, Fer}).

\vspace*{.1in}
A \textbf{contravariant connection} on a vector bundle $E$ over $M$ is a a bilinear map
\[\nabla: \Gamma(T^*M)\times \Gamma(E), \ \ (\alpha, s)\mapsto \nabla_{\alpha}(s)\]
satisfying
\[ \nabla_{f\alpha}(s)= f\nabla_{\alpha}(s), \ \ \nabla_{\alpha}(fs)= f\nabla_{\alpha}(s)+ L_{\alpha}(f) s\]
for all $f\in C^{\infty}(M)$, $\alpha\in \Gamma(T^*M)$, $s\in \Gamma(E)$. The standard operations with connections (duals, tensor products, etc) have an obvious contravariant
version. 

\vspace*{.1in}
A \textbf{cotangent path} (or contravariant path) is a path $a: [0, 1]\rmap T^*M$ sitting above some path $\gamma: [0, 1]\rmap M$, such that
\[ \pi^{\sharp}(a(t))= \frac{d\gamma}{dt}(t).\]
Intuitively, the cotangent path is the pair $(a, \gamma)$ where $\gamma$ is a standard path and the role of $a$ is to encode ``the contravariant derivative of $\gamma$''.
The previous equation says that the contravariant derivative is related to the classical one via $\pi^{\sharp}$.

\vspace*{.1in} Given a contravariant connection $\nabla$ on a vector
bundle $E$, one has a well-defined notion of \textbf{derivative of
sections along cotangent paths}: given a cotangent path $(a,
\gamma)$ and a path $u: [0, 1]\rmap E$ sitting above $\gamma$,
$\nabla_{a}(u)$ is a new path in $E$ sitting above $\gamma$. Writing
$u(t)= s_t(\gamma(t))$ for some time dependent section of $E$,
\[\nabla_{a}(u)= \nabla_{a}(s_t)(x)+ \frac{d  s_t}{dt}(x), \ \ \ \textrm{at}\ x= \gamma(t).\]

\vspace*{.1in}
Given a contravariant connection $\nabla$ on $T^*M$, the \textbf{contravariant torsion} of $\nabla$ is the tensor $T_{\nabla}$ defined by
\[ T_{\nabla}(\alpha, \beta)= \nabla_{\alpha}(\beta)- \nabla_{\beta}(\alpha)- [\alpha, \beta]_{\pi}.\]

\vspace*{.1in} Given a metric $g$ on $T^*M$, one has an associated
\textbf{contravariant Levi-Civita connection}- the unique
contravariant metric connection $\nabla^{g}$ on $T^*M$ whose
contravariant torsion vanishes. The corresponding
\textbf{contravariant geodesics} are defined as usual, as the
(cotangent) curves $a$ satisfying $\nabla_{a}a= 0$. They are the
integral curves of a vector field $\mathcal{V}_{\pi}^{g}$ on $T^*M$,
called the \textbf{contravariant geodesic vector field}. In local
coordinates $(x, y)$ (where $x$ are the coordinates in $M$ and $y$
on the fiber),
\[ \mathcal{V}_{\pi}^g(x, y)= \sum_{p, q} \pi_{p, q}(x)y_p \frac{\partial}{\partial x_q}- \sum \Gamma_{p, q}^{r}(x) y_py_q \frac{\partial}{\partial y_r},\]
where $\Gamma_{p, q}^r$ are the coefficients in
$\nabla_{dx_p}(dx_q)=\sum \Gamma_{p,q}^{r} dx_r$. Geodesics and the
geodesic vector field are actually defined for any contravariant
connection $\nabla$ on $T^*M$, not necessarily of metric type. For
instance, any classical connection $\nabla$ on $T^*M$ induces a
contravariant connection with $\nabla_{\alpha}:=
\nabla_{\pi^{\sharp}\alpha}$ which, in general, is not of metric
type.

Back to our problem, the existence of contravariant sprays is now clear:

\begin{lemma}
Any contravariant geodesic vector field is a contravariant spray.
\end{lemma}

Recall also that (cf. e.g. \cite{CrFe2}) any classical connection $\nabla$ induces two contravariant connections, one on $TM$ and one on $T^*M$, both denoted by
$\overline{\nabla}$:
\[ \overline{\nabla}_{\alpha}(V)= \pi^{\sharp}\nabla_{V}(\alpha)+ [\pi^{\sharp}(\alpha), V], \ \overline{\nabla}_{\alpha}(\beta)= \nabla_{\pi^{\sharp}\beta}(\alpha)+ [\alpha, \beta]_{\pi}.\]
The two are related by the following formula, which follows
immediately from the fact that $\pi^{\sharp}$ is a Lie algebra map
from $(\Omega^1(M), [-, -]_{\pi})$ to the Lie algebra of vector
fields. Note also that this (and its consequences later on) is the
only place where we use that $\pi$ is Poisson.

\begin{lemma} For any classical connection $\nabla$,
\begin{equation}\label{EQ4}
\overline{\nabla}_{\alpha}(\pi^{\sharp}(\beta))=\pi^{\sharp}(\overline{\nabla}_{\alpha}(\beta)).
\end{equation}
\end{lemma}

In the next section we will be using a $\nabla$ which is torsion-free; this condition simplifies the computations
because of the following lemma.

\begin{lemma}\label{ConjConn2} If $\nabla$ is torsion-free then,
for any cotangent path $a$ with base path $\gamma$, for any smooth path $\theta$ in $T^*M$ above $\gamma$ and any smooth path $v$ in $TM$ above
$\gamma$,
\[\langle\overline{\nabla}_{a}(\theta),v\rangle+\langle \theta,\overline{\nabla}_{a}(v)\rangle=\frac{d}{dt}\langle \theta, v\rangle.\]
\end{lemma}

\begin{proof}
Choose a  time-dependent 1-form $A= A(t, x)$ such that $a(t)= A(t,
\gamma(t))$ and similarly a time-dependent 1-form $\Theta$
corresponding to $\theta$ and a time-dependent vector field
corresponding to $v$. Applying the definition of the derivatives
$\overline{\nabla}_{a}$ along cotangent paths and then the
definition of $\overline{\nabla}$ we find that the left hand side at
time $t$ coincides with the following expression on $(t, x)$
evaluated at $x= \gamma(t)$:
\begin{equation}\label{inter}
\langle \nabla_{\pi^{\sharp}\Theta}(A)+ [A, \Theta]_{\pi}+
\frac{d\Theta}{dt}, V\rangle+ \langle \Theta,
\pi^{\sharp}\nabla_{V}(A)+ [\pi^{\sharp}A, V]+ \frac{dV}{dt}\rangle
.
\end{equation}
For the two terms involving $\nabla$ we find
\begin{eqnarray*}
&& \langle \nabla_{\pi^{\sharp}\Theta}(A), V\rangle + \langle \Theta, \pi^{\sharp}\nabla_{V}(A)\rangle =\\
&&\phantom{12}= \langle \nabla_{\pi^{\sharp}\Theta}(A), V\rangle - \langle \nabla_{V}(A), \pi^{\sharp}\Theta \rangle =\\
&&\phantom{12}= L_{\pi^{\sharp}\Theta}\langle A, V \rangle- \langle A, \nabla_{\pi^{\sharp}\Theta}(V) \rangle- L_{V}\langle A, \pi^{\sharp}\Theta \rangle + \langle A, \nabla_{V}(\pi^{\sharp}\Theta)\rangle =\\
&&\phantom{12}= L_{\pi^{\sharp}\Theta}\langle A, V \rangle-
L_{V}\langle A, \pi^{\sharp}\Theta \rangle+ \langle A, [V,
\pi^{\sharp}\Theta]\rangle,
\end{eqnarray*}
where we have used the antisymmetry of $\pi$, then we passed from $\nabla$ on $T^*M$ to its dual on $TM$ and then
we used that $\nabla$ is torsion-free. For the term in (\ref{inter}) containing $[-, -]_{\pi}$, using the definition of this bracket we find
\begin{eqnarray*}
&&  \langle [A, \Theta]_{\pi}, V\rangle= \langle L_{\pi^{\sharp}A}(\Theta)-L_{\pi^{\sharp}\Theta}(A)- d\pi(A, \Theta), V\rangle = \\
&&\phantom{12}= L_{\pi^{\sharp}A}\langle \Theta, V\rangle- \langle
\Theta, [\pi^{\sharp}A, V]\rangle - L_{\pi^{\sharp}\Theta}\langle A,
V\rangle + \langle A, [\pi^{\sharp}\Theta, V]\rangle - L_V(\pi(A,
\Theta)).
\end{eqnarray*}
Plugging the last two expressions into (\ref{inter}) we find
\[ \langle \frac{d\Theta}{dt}, V\rangle + \langle \Theta, \frac{dV}{dt}\rangle+ L_{\pi^{\sharp}A}\langle \Theta, V\rangle .\]
As an expression on $(t, x)$, when evaluated at $x=\gamma(t)$, since
$\pi^{\sharp}A= \frac{d\gamma}{dt}$, we find precisely the right
hand side of the expression from the statement.
\end{proof}

\section{A different formula for $\omega$}

In this section we give another description of $\omega$. The
resulting formula is a generalization of the formula (\ref{omega-0})
from zero's $0_x$ to arbitrary $\xi$'s in $T^*M$. It will depend on
a connection $\nabla$ on $TM$ which is used in order to handle
tangent vectors to $T^*M$. Hence, from now on, we fix such a
connection which we assume to be torsion free. With respect to
$\nabla$, any tangent vector $v\in T_{\xi}(T^*M)$ is determined by
the tangent vector induced on $M$ and by its vertical component
\[ \overline{v}= (dp)_{\xi}(v)\in T_{p(\xi)}M , \ \ \theta_v= v- \textrm{hor}_{\xi}(\overline{v})\in T_{p(\xi)}^{*}M .\]
Of course, when $\xi= 0_x$, these coincide with the components mentioned in the introduction. The fact that
$\nabla$ is torsion-free ensures the following generalization of the formula (\ref{omega-0}) for $\omega_{\textrm{can}}$ at arbitrary $\xi$'s.

\begin{lemma}\label{torsion-free1} If $\nabla$ is torsion-free then, for any $v, w\in T_{\xi}(T^*M)$,
\begin{equation}\label{Omega_can}
\omega_{\mathrm{can}}(v, w)=\langle \theta_w, \overline{v}\rangle-\langle
\theta_v,\overline{w}\rangle,
\end{equation}
\end{lemma}

\begin{proof}  Since $\nabla$ is torsion free, it follows that the associated horizontal distribution $H\subset T(T^*M)$ is Lagrangian with respect to $\omega_{\textrm{can}}$ and then
the formula follows.
\end{proof}

To establish the generalization of (\ref{omega-0}) to arbitrary
$\xi$'s, we start with a tangent vector
\[ v_0\in T_{\xi}(T^*M),\]
with $\xi\in T^*M$ fixed with the property that $\varphi_t(\xi)$ is defined up to $t= 1$. Consider
\[ a: [0, 1]\rmap T^*M, \ a(t)= \varphi_t(\xi) \]
which, from the properties of $\mathcal{V}_{\pi}$, is a cotangent
path; we denote by $\gamma= p\circ a$ its base path. Pushing $v$ by
$\varphi_t$ we obtain a path
\begin{equation}\label{v-t}
t\mapsto v_t:= (\varphi_t)_*(v_0)\in T_{a(t)}(T^*M).
\end{equation}
Taking the components with respect to $\nabla$, we obtain two paths above $\gamma$, one in $TM$ and one in $T^*M$
\[ t\mapsto \overline{v}_t\in T_{\gamma(t)}M, \  \ t\mapsto \theta_{v_t}\in T_{\gamma(t)}^{*}M .\]
We denote these paths by $\overline{v}$ and $\theta_v$. They are related in the following way:

\begin{lemma} $\overline{\nabla}_a \overline{v}=\pi^{\sharp}\theta_{v}$.
\end{lemma}

\begin{proof} We start with one remark on derivatives along vector fields. For any tangent vector $V$ to $\mathcal{U}$ along $a$, $t\mapsto V(t)\in T_{a(t)}\mathcal{U}$,
one has the Lie derivative of $V$ along $\mathcal{V}_{\pi}$, again a tangent vector along $a$, defined by
\[ L_{\mathcal{V}_{\pi}}(V)(t)= \frac{d}{ds}|_{s= 0}  (d\varphi_{-s})_{a(s+t)}(V(s+t)) \in T_{a(t)}\mathcal{U}.\]
We have the following two remarks:
\begin{enumerate}
\item For vertical $V$'s, i.e. coming from a 1-form $\theta$  on $M$ along $\gamma$
$dp(L_{\mathcal{V}_{\pi}}(V))= -\pi^{\sharp}(\theta)$. This follows immediately from the
first property of the spray (e.g. by a local computation).
\item For horizontal $V$'s, $(dp)(L_{\mathcal{V}_{\pi}}(V))= \overline{\nabla}_{a}(\overline{V})$, where $\overline{V}= (dp)(V)$ is a tangent vector to $M$ along $\gamma$. To check this, one may assume that
$V$ is a global vertical vector field on $M$ and one has to show that $(dp)_{\eta}(L_{\mathcal{V}_{\pi}}(V))= \overline{\nabla}_{\eta}(\overline{V})$ for all $\eta\in T^*\mathcal{U}$. Again, this follows immediately by a local computation.
\end{enumerate}
Hence, for an arbitrary $V$ (along $a$), using its components $(\overline{V}, \theta_V)$,
\[ dp(L_{\mathcal{V}_{\pi}}(V)) = -\pi^{\sharp}(\theta_V) + \overline{\nabla}_{a}(\overline{V}).\]
Finally, remark that for our $v$ from (\ref{v-t}), $L_{\mathcal{V}_{\pi}}(v)= 0$.
\end{proof}

We have the following version of (\ref{omega-0}) at arbitrary $\xi$'s in $\mathcal{U}$- a small enough neighborhood
of the zero-section in $T^*M$ on which $\omega$ is well-defined.

\begin{lemma}\label{frmla} Let $\xi\in \mathcal{U}$, $v_0, w_0 \in T_{\xi}(T^*M)$.
Let $v=v_t$ as before and let $\widetilde{\theta}_{v}$ be a path in $T^*M$, solution of the differential equation
\begin{equation}\label{diff-eq}
\nabla_{a}(\widetilde{\theta}_v)= \theta_v .
\end{equation}
Similarly, consider $w= w_t$ and $\widetilde{\theta}_{w}$ corresponding to $w_0$. Then
\[ \omega(v_0, w_0)= (\langle \widetilde{\theta}_w, \overline{v}\rangle - \langle \widetilde{\theta}_v, \overline{w}\rangle- \pi(\widetilde{\theta}_{v},\widetilde{\theta}_w) )|_{0}^{1} .\]
\end{lemma}

\begin{proof}
Since $\nabla$ is torsion-free, Lemma \ref{torsion-free1} implies that
\[ \omega(v_0, w_0)= \int_{0}^{1}(\langle \theta_w, \overline{v}\rangle - \langle \theta_v, \overline{w}\rangle ) dt.\]
Hence it suffices to show that
\[ \langle \theta_w, \overline{v}\rangle - \langle \theta_v, \overline{w}\rangle =
\frac{d}{dt}  (\langle \widetilde{\theta}_w, \overline{v}\rangle - \langle \widetilde{\theta}_v, \overline{w}\rangle- \pi(\widetilde{\theta}_{v},\widetilde{\theta}_w) ).\]
We start from the left hand side, in which we plug in $\theta_v= \nabla_{a}(\widetilde{\theta}_v)$ and the similar formula for $w$, followed by the use of Lemma \ref{ConjConn2}, then the previous lemma, then again the defining formula
for $\widetilde{\theta}_v$; we obtain
\begin{eqnarray*}
&& \frac{d}{dt}(\langle \widetilde{\theta}_w, \overline{v}\rangle - \langle \widetilde{\theta}_v, \overline{w}\rangle)- \langle \widetilde{\theta}_{w}, \overline{\nabla}_{a}(\overline{v})\rangle + \langle \widetilde{\theta}_{v}, \overline{\nabla}_{a}(\overline{w})\rangle = \\
&&\phantom{12}= \frac{d}{dt}(\langle \widetilde{\theta}_w, \overline{v}\rangle - \langle \widetilde{\theta}_v, \overline{w}\rangle)- \langle \widetilde{\theta}_{w}, \pi^{\sharp}(\theta_v)\rangle + \langle \widetilde{\theta}_{v}, \pi^{\sharp}(\theta_w)\rangle= \\
&&\phantom{12}= \frac{d}{dt}(\langle \widetilde{\theta}_w, \overline{v}\rangle - \langle \widetilde{\theta}_v, \overline{w}\rangle)- \langle \widetilde{\theta}_{w}, \pi^{\sharp}(\nabla_{a}(\widetilde{\theta}_v))\rangle + \langle \widetilde{\theta}_{v}, \pi^{\sharp}(\nabla_{a}(\widetilde{\theta}_w))\rangle .
\end{eqnarray*}
For the expression involving the last two term, using
$\pi^{\sharp}(\nabla_{a}(\widetilde{\theta}_v))=
\nabla_{a}(\pi^{\sharp}(\widetilde{\theta}_v))$ and the antisymmetry
of $\pi^{\sharp}$, we find
\[ - \langle \widetilde{\theta}_{w}, \nabla_{a}(\pi^{\sharp}(\widetilde{\theta}_v))\rangle - \langle \overline{\nabla}_a(\widetilde{\theta}_w), \pi^{\sharp}(\widetilde{\theta}_v) \rangle ,\]
which, by Lemma \ref{ConjConn2} again, equals to $-\frac{d}{dt}
\langle \widetilde{\theta}_w,
\pi^{\sharp}(\widetilde{\theta}_v)\rangle= -\frac{d}{dt}
\pi(\widetilde{\theta}_{v},\widetilde{\theta}_w)$. Plugging in the
previous formula, the desired equation follows.
\end{proof}

\section{The proof of the theorem}

We now return to the proof of Theorem \ref{main-theorem}. We start with the proof of the equality
(\ref{main-lemma}) from the introduction. By a dimension counting, it suffices to prove the reverse inclusion. Fix $\xi\in \mathcal{U}$.
We have to show that $\omega(v_0, w_0)= 0$ for all
\begin{equation}\label{cond}
v_0\in \mathcal{F}(p)_{\xi}, \ w_0\in \mathcal{F}(p_1)_{\xi}.
\end{equation}
These conditions are equivalent to $\overline{v}(0)= 0$, $\overline{w}(1)= 0$, where we use the notations from the previous section.
Remark that (\ref{diff-eq}), as an equation on  $\widetilde{\theta}_{v}$, is a linear ordinary differential equation; hence it has solutions
defined for all $t\in [0, 1]$, satisfying any given initial (or final) condition. Hence one may arrange that $\widetilde{\theta}_{v}(0)= 0$, $\widetilde{\theta}_{w}(1)=0$.
The formula from Lemma \ref{frmla} immediately implies that $\omega(v_0, w_0)= 0$.

\vspace*{.2in}

Finally, we show that $p$ is a Poisson map. We have to show that, for $\xi\in \mathcal{U}$ arbitrary,
$\theta\in T_{x}^{*}M$  ($x= p(\xi)$), the unique $v_0\in T_{\xi}\mathcal{U}$ satisfying
\begin{equation}\label{Poisson_map}
 p^{*}(\theta)_{\xi}= i_{v_0}(\omega)
\end{equation}
also satisfies $(dp)_{\xi}(v_0)= \pi^{\sharp}(\theta)$. From the
previous formula it immediately follows that $v_0$ is in
$\mathcal{F}(p)^{\perp}$, hence in $\mathcal{F}(p_1)$, hence
$\overline{v}(1)= 0$, where we start using the notations from Lemma
\ref{frmla}. Next, we evaluate (\ref{Poisson_map}) on an arbitrary
$w_0\in T_{\xi}\mathcal{U}$. We also use the formula for $\omega$
from Lemma \ref{frmla}, where $\widetilde{\theta}_v$ and
$\widetilde{\theta}_w$ are chosen so that $\widetilde{\theta}_v(1)=
0$ and $\widetilde{\theta}_w(0)= \eta\in T_{x}^{*}M$ is arbitrary.
We find:
\[ \theta(\overline{w}_0)= \langle \widetilde{\theta}_v(0), \overline{w}_0\rangle + \langle \eta, \pi^{\sharp}\widetilde{\theta}_v(0)- \overline{v}_0\rangle\]
Since this holds for all $w_0$ and all $\eta$, we deduce that
$\theta= \widetilde{\theta}_v(0)$,
$\pi^{\sharp}\widetilde{\theta}_v(0)= \overline{v}_0$. Hence
$\pi^{\sharp}(\theta)= \overline{v}_0= (dp)_{\xi}(v_0)$.

\section{Some remarks}

Here are some remarks on possible variations. First of all, regarding the notion of contravariant spray, the first condition means that, locally, $\mathcal{V}_{\pi}$ is
of type
\[ \mathcal{V}_{\pi}(x, y)= \sum_{p, q} \pi_{p, q}(x, y)y_p\frac{\partial}{\partial x_q}+ \sum_i \gamma^i(x, y) \frac{\partial}{\partial y_i}.\]
The second condition means that each $\gamma^i(x, y)$ is of type $\sum_{j, k} \gamma^{i}_{j, k}(x) y_j y_k$.
While the first condition has been heavily used in the paper, the second one was
only used to ensure that $\omega$ is well-defined and non-degenerate at elements $0_x\in T^{*}_{x}M$.

Another remark is that one can show that $\mathcal{U}$ can be made into a local symplectic groupoid, with source map $p$ and target map $p_1$;
see also \cite{Kar}.

Let us also point out where we used that $\pi$ is Poisson: it is
only for the compatibility relation (\ref{EQ4}) which, in turn, was
only used at the end of the proof of the Lemma \ref{frmla}. However,
it is easy to keep track of the extra-terms that show up for general
bivectors $\pi$: at the right hand side of (\ref{EQ4}) one has to
add the term $i_{\alpha\wedge \beta}(\chi_{\pi})$ where $\chi_{\pi}=
[\pi, \pi]$, while to the equation from Lemma \ref{frmla} the term
$\int_{0}^{1} \chi_{\pi}(a, \widetilde{\theta}_v,
\widetilde{\theta}_w) dt$. This is useful e.g. for handling various
twisted versions. E.g., for a  $\sigma$-twisted bivector $\pi$ on
$M$ in the sense of \cite{WeinSev} (i.e. satisfying $[\pi, \pi]=
\pi^{\sharp}(\sigma)$ where $\sigma$ is a given closed 3-form
$\sigma$ on $M$), the interesting (twisted symplectic) 2-form on
$\mathcal{U}$ is the previously defined $\omega$ to which we add the
new two-form $\omega_{\sigma}$ given by (compare with \cite{CaXu}):
\[ \omega_{\sigma} = \int_{0}^{1} \varphi_{t}^{*}(i_{\mathcal{V}_{\pi}}p^*(\sigma)) dt.\]

\bibliographystyle{amsplain}
\def\lllll{}

\end{document}